\documentclass[12pt]{amsart}
\usepackage{amsmath}
\usepackage{amssymb}
\usepackage{mathrsfs}
\usepackage{geometry}
\usepackage{graphicx}
\usepackage{amsmath}
\usepackage{hyperref}
\usepackage{amsthm}
\usepackage{color}
\usepackage{mathtools}
\usepackage{geometry}

\usepackage{amsthm}
\newtheorem{theorem}{Theorem}
\newtheorem{thm}{Theorem}

\newtheorem{lemma}[theorem]{Lemma}

\DeclareMathOperator{\arcsinh}{arsinh}

\DeclareMathOperator{\discover}{\mathsf{D}}     
\DeclareMathOperator{\algorithm}{\mathsf{T}}   
\DeclareMathOperator{\area}{area_{\mathbb{H}^2}}     
\DeclareMathOperator{\var}{Var}                          
\DeclareMathOperator{\influence}{Inf}                 
\DeclareMathOperator{\ann}{Ann_p}                     
\DeclareMathOperator{\annsec}{N_p}                   %
\DeclareMathOperator{\annsim}{N}                      
\DeclareMathOperator{\ball}{B}                            
\DeclareMathOperator{\sph}{\partial B}                
\DeclareMathOperator{\cell}{C}         

\begin{document}
\title[Sharp phase transition for Voronoi percolation in $\mathbb{H}^d$]{Sharpness of phase transition for Voronoi percolation in hyperbolic space}

\author{Xinyi Li}
\address[Xinyi Li]{Beijing International Center for Mathematical Research, Peking University}
\email{xinyili@bicmr.pku.edu.cn}

\author{Yu Liu}
\address[Yu Liu]{School of Mathematical Sciences, Peking University}
\email{liuyu8300@stu.pku.edu.cn}

\maketitle 
    
    \begin{abstract}
    	In this paper, we consider Voronoi percolation in  the hyperbolic space $\mathbb{H}^d$ ($d\ge 2$) and show that the phase transition is sharp. More precisely, we show that for Voronoi percolation with parameter $p$ generated by a homogeneous Poisson point process with intensity $\lambda$,
    	there exists $p_c:=p_c(\lambda,d)$ such that the probability of a monochromatic path from the origin reaching a distance of $n$ decays exponentially fast in $n$. We also prove the mean-field lower bound $\mathbb{P}_{\lambda,p}(0\leftrightarrow \infty)\ge c(p-p_c)$ for $p>p_c$.

    \end{abstract}

    \section{Introduction}

    Percolation theory was first introduced by Broadbent and Hammersley in \cite{bernoulli} as a mathematical model for the study of liquids or gas penetrating through porous media. Since then, with a wide range of open questions and many applications in statistical physics, it has received considerable scholarly attention. We refer the reader to \cite{grimmett2} and \cite{bpercolation} for a general introduction to percolation theory.
    
    The most fundamental question in the study of percolation is the existence of phase transition, that is to say, the sheer difference of the behaviour of the model between subcritical and supercritical regimes. In answering this question, Aizenman and Barsky \cite{AB87} and Menshikov \cite{Men86} gave the first proofs of sharp phase transition for Bernoulli percolation on $\mathbb{Z}^d$. More precisely, they showed that there is the exponential decay of the probability of one-arm events in the subcritical regime and the mean-field lower bound in the supercritical phase. Later, more proofs in this direction emerged; see e.g. \cite{grimmett2}, \cite{history2} for more thorough discussion on this matter.

    In parallel with the growing understanding of discrete percolation models, many works have been dedicated to the study of continuum percolation, including Poisson boolean model and Voronoi model, where new challenges arise due to spatial dependencies. 
    We refer readers to \cite{MR96} for more on continuum percolation.
    Recently, more properties concerning continuum model such as noise sensitivity (\cite{ABGM14}, \cite{AB18}), sharp phase transition of Poisson boolean model (\cite{ATT17}) and rate of convergence for the crossing probability of quenched Voronoi percolation (\cite{AGMT16}, \cite{ADG21}) have also been developed.

    One of the most important examples of the continuum percolation model is Poisson Voronoi percolation, in which one colours the Voronoi cells generated by a Poisson point process in $\mathbb{R}^d$  black with probability $p$ and white with probability $(1-p)$, independently of the colours of all other Voronoi cells.
    In \cite{BR06}, Bollob{\'a}s and Riordan proved that the planar Voronoi percolation undergoes a sharp phase transition at $p=1/2$, and the exponential decay of connection probability in subcritical regime in $\mathbb{R}^d$ for every $d\ge 2$ was studied in \cite{sharpv}.

    The Voronoi percolation model can also be constructed in the hyperbolic space $\mathbb{H}^d$ (see \cite{BS01}). Here, different from the Euclidean case, the intensity $\lambda$ of the Poisson point process in $\mathbb{H}^d$ does matter, and we denote the law of the random colouring of $\mathbb{H}^d$ by $\mathbb{P}_{\lambda,p}$. For Voronoi percolation in hyperbolic spaces, Hansen and M{\"u}ller proved that the critical probability for the existence of an infinite cluster tends to $1/2$ as the intensity of the Poisson point process $\lambda$ tends to infinity in \cite{hbvoronoi} and asymptotically equals to ${\pi \lambda}/{3}$ as $\lambda$ tends to 0 in \cite{hansennew}. 
    In this paper we are going to offer a proof of sharpness of phase transition of Voronoi percolation on $\mathbb{H}^d$. Instead of giving the formal definition here, we first roughly describe our main result of this paper:
    \begin{thm}
    	\label{thm1}
    	For $d\ge 2$ and $\lambda>0$, there exists $p_c:=p_c(\lambda,d)$ such that:\\
    	1. For every $p<p_c$, there exists a constant $c_p:=c_p(\lambda,d)>0$ such that 
    	$$\mathbb{P}_{\lambda,p}(0\text{\ is\ connected\ to\ distance\ }n \text{\ by\ a\ black\ path})\le \exp(-c_pn).$$
    	2. There exists $c=c(\lambda,d)>0$ such that for all $p>p_c$, 
    	$$\mathbb{P}_{\lambda,p}(0\text{\ is\ contained\ in an\ unbounded\ black connected component})\ge c(p-p_c).$$
    \end{thm}
    
    While extending existing proofs of sharp phase transition for discrete models to this model is not easy, a new approach making use of randomized algorithms and the OSSS inequality introduced in \cite{osss}, has been developed in \cite{sharpr}.
    In this work, we also follow the line of \cite{sharpv}, adapting the arguments to the hyperbolic case. 
    
    Finally, let us remark that in \cite{stoppingsets}, the authors offered a new version of the OSSS inequality for functionals of a general Poisson process, providing another plausibly feasible approach of Theorem \ref{thm1}. However, in this paper we still chose to discretize the hyperbolic space and apply the original OSSS inequality as it is enough to provide a concise proof of our main result.

    This work is organized as follows: in Section \ref{sec:2} we introduce basic notations and give a few preliminary results that will be used in proof of Theorem \ref{thm1} which is wrapped up in Section \ref{sec:3}

\vspace{3mm}
{\bf Acknowledgments}: The research of the authors is supported by NSFC (No.\ 12071012) and the National Key R\&D Program of China (No.\ 2020YFA0712900).
	
	\section{Notation and preliminaries}\label{sec:2}
	
	\subsection{Hyperbolic space}
	
	Hyperbolic $d-$space ($d\ge 2$), denoted by $\mathbb{H}^d$, is the maximally symmetric, simply connected, $d-$dimensional Riemannian manifold with a constant negative sectional curvature.
	Typically, people study hyperbolic space by constructing models in order to utilize the Euclidean space to describe the hyperbolic space.
	In our following arguments, we choose the Poincar\'e ball model to represent $\mathbb{H}^d$ and now we are going to provide a brief summary of its definitions.
	
	The Poincar\'e ball model can be constructed through equipping the open unit ball $\mathcal{B}^d=\{x\in \mathbb{R}^d:\Vert x\Vert <1\}$ with a hyperbolic distance $d_p$ and a metric tensor. Here, $\Vert \cdot\Vert $ represents the Euclidean norm.
	
	For points $x,y\in \mathcal{B}^d$, the hyperbolic distance is defined by
	$$
	d_p(x,y)\coloneqq 2\arcsinh\left(\frac{\Vert x-y\Vert }{\sqrt{(1-\Vert x\Vert ^2)(1-\Vert y\Vert ^2)}}\right),
	$$
	where $\arcsinh$ is the inverse function of hyperbolic sine. 	
	
	The metric tensor in the Poincar\'e ball model is 
	$$\text{d}s^2=\frac{4\Vert \text{d}x\Vert ^2}{(1-\Vert x\Vert ^2)^2}.$$

	
	In this paper we will consider both Euclidean balls and hyperbolic balls,  and thus in the following context we will use subscripts to distinguish them.
	For example, we use $\ball_p(x,r)$ and $\ball_e(x,r)$ respectively to represent the hyperbolic ball and the Euclidean ball centered at $x$ with radius $r$. 
	Similarly, the hyperbolic sphere centered at $x$ with radius $r$ is denoted by $\sph_p(x,r)$.
	Also notice that every hyperbolic sphere centered at the origin is also an Euclidean sphere centered at the origin. More precisely, 
	$$\sph_p(0,r)=\sph_e(0,\tanh(\frac{r}{2})),\ \ r>0.$$
	
	And we denote the hyperbolic annulus centered at the origin with inner radius $r$ and outer radius $R$ while containing the inner boundary by 
	$$\ann(r,R)\coloneqq\{y\in \mathbb{H}^d: r\le d_p(0,y)<R\}.$$
	
	For a countable set of points $\mathcal{Z}$ in the hyperbolic plane, we denote the corresponding hyperbolic Voronoi cell of $z\in \mathcal{Z}$ by
	$$\cell_p(z;\mathcal{Z})\coloneqq\{y\in \mathbb{H}^d: d_p(y,z)=\inf_{z'\in \mathcal{Z}} d_p(y,z')\}.$$
	

	\subsection{Hyperbolic Poisson point process and Voronoi percolation}

	A homogeneous Poisson point process with constant intensity $\lambda$ on $\mathbb{H}^d$ can be viewed as an inhomogeneous Poisson point process on the ordinary Euclidean space with intensity
	$$u\mapsto \lambda\cdot \mathbf{1}_{\mathbb{D}}(u)\frac{4}{(1-\Vert u \Vert^2)^2}.$$

	A gentle introduction to point process and Poisson point process can be found in \cite{PPP} and in the rest of this paper we will usually use $\mathcal{Z}$ to denote a homogeneous Poisson point process with constant intensity $\lambda$ on $\mathbb{H}^d$. 
	
	
	After presenting the definition of Poisson point process on $\mathbb{H}^d$, we attach to each point in $\mathcal{Z}$ a randomly chosen colour independently. Precisely, every point in $\mathcal{Z}$ is coloured black with probability $p$ and white with probability $1-p$, and the colours of different points in $\mathcal{Z}$ are independent. By using $\mathcal{Z}_b$ (resp. $\mathcal{Z}_w$) to denote the black points (resp. white points) after the colouring process, we can also view $\mathcal{Z}_b$ and $\mathcal{Z}_w$ as independent Poisson point processes on $\mathbb{H}^d$ with respective intensities $\lambda\cdot p$ and $\lambda\cdot (1-p)$. 	
	Let $\mathbb{P}_{\lambda,p}$ denote the law of $(\mathcal{Z}_b,\mathcal{Z}_w)$, or equivalently the law of $\mathcal{Z}$ with its colours, then
	the measure $\mathbb{P}_{\lambda,p}$ induces a colouring $\omega$ of the hyperbolic space $\mathbb{H}^d$ defined as follows:
	$$\omega(y)=1, \exists z\in \mathcal{Z}_b,\ y\in \cell_p(z;\mathcal{Z}),\quad 
		\omega(y)=0, \text{otherwise}.$$

	From now on we will use $\mathbb{P}_p$ for $\mathbb{P}_{\lambda,p}$ when it does not cause misunderstanding since we are going to prove Theorem \ref{thm1} for every fixed $\lambda>0$.
	
	For $x,y\in \mathbb{H}^d$, let $\{x\leftrightarrow y\}$ denote the event that there exists a continuous path of black points connecting $x$ to $y$, and let $\{0\leftrightarrow \infty\}$ denotes the event that the origin belongs to a infinite-volume connected component of black points.
	For two subsets $A$ and $B$ of $\mathbb{H}^d$, we use $\{A\leftrightarrow B\}$ to denote $\{\exists\ x\in A,\ y\in B \text{\ such\ that }x\leftrightarrow y\}$, and briefly write $\{x\leftrightarrow B\}$ for $\{\{x\}\leftrightarrow B\}$.
	Also for $p\in [0,1]$, $\lambda>0$, $n\ge 0$, define
	$$\theta(p)\coloneqq\mathbb{P}_{p}(0\leftrightarrow \infty),\quad 
	\theta_n(p)\coloneqq\mathbb{P}_{p}(0\leftrightarrow \sph_p(0,n)),\quad 
	p_c\coloneqq\inf\{p\in [0,1]: \theta(p)>0\}.$$
	Then we can abbreviate $\{0\ \text{is\ contained\ in\ an\ unbounded\ black\ connected\ component}\}$ as $\{0 \leftrightarrow \infty\}$ and $\{0\  \text{is\ contained\ to\ distance\ } n\ \text{by\ a\ black\ path}\}$ as $\{0\leftrightarrow \sph_p(0,n)\}$.
	%
	%
	%
	%
	
	\subsection{Increasing events and the FKG inequality}
	
	An event $A$ is said to be $increasing$ if for every configurations $(\mathcal{Z}_b,\mathcal{Z}_w)$ and $(\mathcal{\tilde{Z}}_b,\tilde{\mathcal{Z}}_w)$, 
	$$
	\left.
	\begin{aligned}
	&(\mathcal{Z}_b,\mathcal{Z}_w)\in A,\\
	&\mathcal{Z}_b \subseteq \mathcal{\tilde{Z}}_b,\\
	&\tilde{\mathcal{Z}}_w\subseteq \mathcal{Z}_w
	\end{aligned}
	\right\}
	\Rightarrow (\mathcal{\tilde{Z}}_b,\tilde{\mathcal{Z}}_w)\in A$$
	
	An event $W$ is said to be $decreasing$ if its complement $W^c$ is increasing.

	As in Bernoulli percolation model, we also have  FKG inequality for Voronoi percolation in $\mathbb{H}^d$ listed in the following context. We will avoid the proof and refer the reader to \cite[Chapter 8]{bpercolation} for more information about this lemma.
	
	\begin{lemma}[\cite{bpercolation},Lemma 14]
		Let $A$ and $B$ be two increasing events. Then for any $\lambda>0$, $0\le p\le 1$,
		\begin{equation}\label{FKG}
			\mathbb{P}_{p}(A\cap B)\ge \mathbb{P}_{p}(A)\mathbb{P}_{p}(B).
		\end{equation}
	\end{lemma}

	\subsection{A Russo-type formula}
	
	For an increasing event $A$ and a configuration $(\mathcal{Z}_b,\mathcal{Z}_w)$, we define the set of $pivotal\  points$ by
	$$\mathsf{Piv}_A(\mathcal{Z}_b,\mathcal{Z}_w)\coloneqq\{x\in \mathcal{Z}:\mathbf{1}_A(\mathcal{Z}_b\setminus \{x\},\mathcal{Z}_w\cup \{x\})\neq \mathbf{1}_A(\mathcal{Z}_b\cup \{x\},\mathcal{Z}_w\setminus \{x\})\}.$$
	
	And define an increasing event $A$ $local$ if there exists $n\ge 0$ such that $A$ is measurable with respect to the $\sigma$-field generated by $\{\omega(x)\}_{x\in \ball_p(0,n)}$.

	\begin{lemma}
		\label{russo}
		The function $p\rightarrow \mathbb{P}_{p}(A)$ is differentiable if $A$ is a local increasing event and 
		\begin{equation}
			\frac{\mathrm{d}\mathbb{P}_{p}(A)}{\mathrm{d}p}=\mathbb{E}_{p}(|\mathsf{Piv}A|).
		\end{equation}
	\end{lemma}
	
	The main idea of this lemma can be seen in [\cite{sharpv}, Lemma 4]. Here, we supplement the details for the proof of the integrability of $|\mathsf{Piv}A|$ and the continuity of $\mathbb{E}_s(|\mathsf{Piv}_A|)$ which is not mentioned in the former paper.
	
	\begin{proof}[Proof of Lemma 2]
		
		We denote the law of $\mathcal{Z}$ by d$\mathcal{Z}$ and write
		$$\mathbb{P}_{p+\delta}(A)-\mathbb{P}_p(A)=\int_{\mathcal{Z}}\mathbb{P}_{p+\delta}(A|\mathcal{Z})-\mathbb{P}_p(A|\mathcal{Z})\text{d}\mathcal{Z}. $$
		Notice that when $\mathcal{Z}$ is given, the law of $\mathcal{Z}_b$ is Bernoulli percolation with parameter $p$ on all points of $\mathcal{Z}$. Assume that $A$ only depends on the $\sigma$-field generated by $\{\omega(x)\}_{x\in \ball_p(0,n)}$ for a fixed $n$. Then we can apply Russo's formula for Bernoulli percolation(\cite[Section 2.4]{grimmett2}) to get the differentiability of $\mathbb{P}_p(A|\mathcal{Z})$ with respect to $p$ and 
		$$\frac{\text{d}\mathbb{P}_p(A|\mathcal{Z})}{\text{d}p}=\sum_{x\in \mathcal{Z}}\mathbb{E}_p(\mathbf{1}_{x\in \mathsf{Piv}_A}|\mathcal{Z}).$$
		Therefore, combining Fubini's gives 
		\begin{equation*}
			\begin{aligned}
			\mathbb{P}_{p+\delta}(A)-\mathbb{P}_p(A)
			&=\int_{\mathcal{Z}}\int_p^{p+\delta}\mathbb{P}_s'(A|\mathcal{Z})\text{d}s\text{d}\mathcal{Z}
			\ \ \  =\int_{\mathcal{Z}}\int_p^{p+\delta}\mathbb{E}_s(|\mathsf{Piv}_A\Vert \mathcal{Z})\text{d}s\text{d}\mathcal{Z}\\
			&=\int_p^{p+\delta}\!\!\!\!  \int_{\mathcal{Z}}\mathbb{E}_s(|\mathsf{Piv}_A\Vert \mathcal{Z})\text{d}\mathcal{Z}\text{d}s
			=\int_p^{p+\delta}\mathbb{E}_s(|\mathsf{Piv}_A|)\text{d}s.
			\end{aligned}
		\end{equation*}
		
		Therefore, it suffices to explain that $|\mathsf{Piv}_A|$ is integrable and that $\mathbb{E}_s(|\mathsf{Piv}_A|)$ is continuous with respect to $s$. 
		
		For the former question, let $D_m(\mathcal{Z})\subset \mathcal{Z}$ denote the set of points in $\mathcal{Z}$ whose cells intersect the ball $\ball_p(0,m)$. Then the locality of  $A$ implies
		\begin{equation}\label{piv}
			|\mathsf{Piv}_A|\le |D_n(\mathcal{Z})|.
		\end{equation}
		
		First notice that by standard estimates of the Poisson-Voronoi tessellation, there exists $c_1>0$ such that for every $t\ge n$,
		\begin{equation}\label{4times}   
			\mathbb{P}_p(D_n(\mathcal{Z})\cap \ball_p(0,4t)^c\neq \varnothing)\le \mathbb{P}_p(\mathcal{Z}\cap \ball_p(0,t)=\varnothing)\le \exp(-c_1t^d).
		\end{equation}
		
		Therefore, combining \eqref{piv} \eqref{4times} and separating $D_n(\mathcal{Z})$ into disjoint annuli gives
		$$\begin{aligned}
		&\quad\  \mathbb{E}_s(|D_n(\mathcal{Z})|)\\
		&=\mathbb{E}_s(|D_n(\mathcal{Z})\cap \ball_p(0,4n)|)+\sum_{k=4n}^{\infty}\mathbb{E}_s(|D_n(\mathcal{Z})\cap \ann(k,k+1)|)\\
		&\le \mathbb{E}_s(|D_n(\mathcal{Z})\cap \ball_p(0,4n)|)+\sum_{k=4n}^{\infty}\mathbb{P}_s(\mathcal{Z}\cap \ball_p(0,\frac{k}{4})=\varnothing)\mathbb{E}_s(|\mathcal{Z}\cap \ann(k,k+1)|)<+\infty.
		\end{aligned}$$
		
		As for the second part, notice that once we prove the continuity of $\mathbb{E}_s(|\mathsf{Piv}_A\cap \ball_p(0,4n)|)$ and $\mathbb{E}_s(|\mathsf{Piv}_A\cap \ann(k,k+1)|)$, $k\ge 4n$, the continuity of 
		$\mathbb{E}_s(|\mathsf{Piv_A}|)$ with respect to $s$ then follows from the uniform convergence of 
		$$\mathbb{E}_s(|\mathsf{Piv}_A\cap \ball_p(0,4n)|)+\sum_{k=4n}^{\infty}\mathbb{E}_s(|\mathsf{Piv}_A\cap \ann(k,k+1)|).$$
		
		For every fixed $n\ge 1$, for every $\epsilon>0$, using the integrability of $\mathbb{E}_s(|\mathcal{Z}\cap \ball_p(0,4n)|)$ we can choose large $K$ and $\delta>0$ such that when $|t-s|<\delta$, 
		$$\begin{aligned}
		&\quad\ |\mathbb{E}_t(|\mathsf{Piv}_A\cap \ball_p(0,4n)|)-\mathbb{E}_s(|\mathsf{Piv}_A\cap \ball_p(0,4n)|)|\\
		&\le \int_{\{ |\mathcal{Z}\cap \ball_p(0,4n)|\le K\}}
		|\mathbb{E}_t(|\mathsf{Piv}_A\cap \ball_p(0,4n)||\mathcal{Z})-\mathbb{E}_s(|\mathsf{Piv}_A\cap \ball_p(0,4n)||\mathcal{Z})|\text{d}\mathcal{Z}\\
		&+\sum_{H=K}^{+\infty}\int_{\{|\mathcal{Z}\cap \ball_p(0,4n)|=H\}}|\mathbb{E}_t(|\mathsf{Piv}_A\cap \ball_p(0,4n)||\mathcal{Z})-\mathbb{E}_s(|\mathsf{Piv}_A\cap \ball_p(0,4n)||\mathcal{Z})|\text{d}\mathcal{Z}\\
		&\le K[1-(1-|t-s|)^K]+\mathbb{E}(|\mathcal{Z}\cap \ball_p(0,4n)|\mathbf{1}_{\{|\mathcal{Z}\cap \ball_p(0,4n)|\ge K\}})<\epsilon.
		\end{aligned}$$
	
		Similarly, we can derive the continuity of $\mathbb{E}_s(|\mathsf{Piv}_A\cap \ann(k,k+1)|)$ for every $k\ge 4n$, and therefore conclude that $\mathbb{E}_s(|\mathsf{Piv_A}|)$ is continuous with respect to $s$.
		
	\end{proof}
	
	\subsection{The OSSS inequality}
	
	The OSSS inequality was introduced in \cite{osss} and states that for a given boolean function $f$ and a randomized algorithm $\algorithm$ determining $f$, the variance of this boolean function can be controlled by the revealments of each coordinate during the process of $\algorithm$ and the influences of coordinates. 
	
	Assume that $I$ is a countable set, $(\Omega^I,\pi^{\otimes I})$ is a product probability space and $f$ is a function mapping $\Omega^I$ to  $\{0,1\}$.
	A decision tree $\algorithm$ determining $f$ is an algorithm that takes $\omega=(\omega_i)_{i\in I}\in \Omega^I$ as an input, and reveals the coordinates of $\omega$ step by step based on the values of $\omega$ revealed so far. The algorithm stops when the possible values of $f$ no longer depend on the values of the coordinates of $\omega$ that have not been revealed.
		
	The revealment of the $i$-th coordinate is defined by 
	$$\delta_i(\algorithm):=\pi^{\otimes I}(\omega: \algorithm\ \text{reveals\ the\ value\ of\ }\omega_i),$$
	and the influence of the $i$-th coordinate is given by
	$$\influence_i(f):=\pi^{\otimes I}(\omega:f(\omega)\neq f(\tilde{\omega})),$$
	where $\tilde{\omega}$ represents the random element in $\Omega^I$ which is the same as $\omega$ in every coordinate while has the $i$-th coordinate resampled independently.
	
	The OSSS inequality is stated as follows.
	\begin{lemma}[OSSS]
		\label{OSSS}	
		For any function $f:\Omega^I\rightarrow \{0,1\}$ and any algorithm $\algorithm$ determining $f$, which stops in finite steps almost surely, we have
		\begin{equation}\label{osss}
			\var(f)\le \sum_{i\in I}\delta_i(\algorithm)\influence_i(f).
		\end{equation}
	\end{lemma}
	
	\section{Proofs}\label{sec:3}
	
	\subsection{Discretization of Voronoi percolation}
	
	Before turning back to the proof of our main theorem, we first introduce an appropriate product space to encode the measure of Voronoi percolation on the Poincar\'e disk model so that we are able to use Lemma \ref{OSSS}. Here, to present the idea clearly, we will only focus on the case for $d=2$. And as for higher dimensions, using the polar coordinate to denote the volume of a measurable set and cutting the angles appropriately just like what we are going to do for $d=2$, we can also derive Lemma \ref{lem4} and give the proof of the main theorem.
	
	First note that for $d=2$ and every measurable subset $B\subset \mathbb{H}^2$, the area is given by
	$$\area(B)=\int\!\!\! \int_B \frac{4}{(1-x^2-y^2)}\text{d}x\text{d}y.$$
	
	Now we are going to divide the hyperbolic space $\mathbb{H}^2$ up into countable disjoint hyperbolic annuli centered at the origin and cut each annulus into several sectors while making sure that the areas of these sectors are uniformly bounded.
	
	Precisely, for a given $\epsilon>0$, we first regard $\mathbb{H}^2$ as unions of annuli of the form $\ann(2k\epsilon, 2(k+1)\epsilon)$, $k\ge 0$.
	Define $N_k^{\epsilon}=[\frac{\sinh((2k+1)\epsilon)}{\sinh\epsilon}]+1$. 
	For each $k\ge 0$, we then divide $\ann(2k\epsilon, 2(k+1)\epsilon)$ into $N_k^{\epsilon}$ annulus sectors of the form
	$$\annsec(2k\epsilon, 2(k+1)\epsilon, \theta_1, \theta_2)=\{re^{i\theta}:2k\epsilon\le d_p(re^{i\theta},0)<2(k+1)\epsilon, \theta \in [\theta_1,\theta_2)\},$$ 
	while making sure that each two of these sectors are disjoint and congruent. 
	
	Therefore, there exists a determined constant $C$ such that for every $\epsilon>0$ and each $k\ge 0$, 
	$$\area(\annsec(2k\epsilon, 2(k+1)\epsilon, \theta_1, \theta_2))\le  C(\sinh\epsilon)^2.$$
	
	In order to denote annulus sectors clearly, we introduce the definition of representative points:
	
	\begin{enumerate}
		\item For $k=0$, let $x_0^{\epsilon}=0$ be the representative point of the degenerate annulus $\ann(0, 2\epsilon)=\ball_p(0, 2\epsilon)$. 
		
		\item For $k\ge 1$, define the representative point of $\annsec(2k\epsilon, 2(k+1)\epsilon, \theta_1, \theta_2)$ as the exact point $re^{i\theta_1}$ satisfying $d_p(re^{i\theta_1},0)=2k\epsilon$. And denote the representative points contained in $\ann(2k\epsilon, 2(k+1)\epsilon)$ by $x_{k,l}^{\epsilon}$, $l=1, 2, \cdots, N_k^{\epsilon}$. 
	\end{enumerate}

	Let $K_{\epsilon}$ be the set of all representative points, i.e.,
	$K_{\epsilon}=\cup_{k=1}^{\infty}\cup_{l=1}^{N_k^{\epsilon}}\{x_{k,l}^{\epsilon}\}\cup\{x_0^{\epsilon}\}.$
	Then for every annulus sector $\annsec(2k\epsilon, 2(k+1)\epsilon, \theta_1, \theta_2)$, there exists a unique representative point, say, $x_{k,l}^{\epsilon}$ contained in this annulus sector. We then denote this region by $\annsim_{x_{k,l}^{\epsilon}}$ for simplicity.
	
	
		
	For every $x\in K_{\epsilon}$, define 
	$$\mathcal{Z}^b_x=\mathcal{Z}^b\cap \annsim_x,\ \mathcal{Z}^w_x=\mathcal{Z}^w\cap \annsim_x,\ \text{and}\   \mathcal{Z}_x=\mathcal{Z}^w_x\cup \mathcal{Z}^b_x.$$
	
	Let $(\Omega_x, \pi_x)$ be the space associated with the random variable $\mathcal{Z}_x=(\mathcal{Z}_x^b,\mathcal{Z}_x^w)$, and it is obvious that the random variable $\mathcal{Z}_x$ are independent for different $x$.
	Then the product space $(\prod_{x\in K_{\epsilon}}\annsim_x,\otimes_{x\in K_{\epsilon}}\pi_x)$ agrees with the original space and we will thus use $\mathbb{P}_{p}$ to denote the probability measure for the product space.

	For every $x\in K_{\epsilon}$ and every increasing event $A$, we define 
	$$\influence^{\epsilon}_x(A)=\mathbb{P}_{p}(\mathbf{1}_A(\mathcal{Z})\neq\mathbf{1}_A(\tilde{\mathcal{Z}})),$$
	where $\mathcal{Z}=(\mathcal{Z}_x)_{x\in K_{\epsilon}}$ has law $\otimes_{x\in K_{\epsilon}}\pi_x$, and $\tilde{\mathcal{Z}}$ is a random variable which equals to $\mathcal{Z}$ except on $x-$coordinate which is resampled independently.

	\begin{lemma}
		\label{lem4}
		For every local increasing event $A$, 
		\begin{equation}\label{derivative}
			\frac{\mathrm{d}\mathbb{P}_p(A)}{\mathrm{d}p}\ge \frac{1}{2}\limsup_{\epsilon\rightarrow 0}\sum_{x\in K_{\epsilon}}\influence^{\epsilon}_x(A).
		\end{equation}
	\end{lemma}
	
	\begin{proof}[Proof of Lemma 4]
		
		For a fixed local increasing event $A$, we assume $A$ only depends on colours in $\ball_p(0,n)$. First we prove that for any $m\ge 1$, 
		$$ 
		\frac{\text{d}\mathbb{P}_p(A)}{\text{d}p}\ge \frac{1}{2}\limsup_{\epsilon\rightarrow 0}\sum_{x\in K_{\epsilon}\cap \ball_p(0,m)}\influence^{\epsilon}_x(A).$$
			
		For a fixed $m\ge 1$, every $x\in K_{\epsilon}\cap \ball_p(0,m)$, observe that  $\mathcal{Z}_x=\tilde{\mathcal{Z}}_x=\varnothing$ leads to $\mathcal{Z}=\tilde{\mathcal{Z}}$ and  $\mathbf{1}_A(\mathcal{Z})= \mathbf{1}_A(\tilde{\mathcal{Z}})$. Combining with the definition of Poisson point process we get that with probability $O({\epsilon}^4)$, $\mathcal{Z}_x\cup\tilde{\mathcal{Z}}_x$ contains more than one point, then
		$$\begin{aligned}
			&\ \mathbb{P}_{p}(\mathbf{1}_A(\mathcal{Z})\neq \mathbf{1}_A(\tilde{\mathcal{Z}}))\\
			=&\ 2\mathbb{P}_{p}(\mathbf{1}_A(\mathcal{Z})\neq \mathbf{1}_A(\tilde{\mathcal{Z}}),|\mathcal{Z}_x|=1,|\tilde{\mathcal{Z}}_x|=0)+\mathbb{P}_{p}(\mathbf{1}_A(\mathcal{Z})\neq \mathbf{1}_A(\tilde{\mathcal{Z}}),|\mathcal{Z}_x|+|\tilde{\mathcal{Z}}_x|\ge 2)\\
			\le&\ 2\mathbb{P}_{p}(|\mathsf{Piv_A}\cap \annsim_x|\ge 1, |\mathcal{Z}_x|=1,|\tilde{\mathcal{Z}}_x|=0)+O({\epsilon}^4)\\
			\le&\ 2\mathbb{E}_{p}(|\mathsf{Piv_A}\cap \annsim_x|)+O({\epsilon}^4).
		\end{aligned}$$
		
		Summing this inequality over all points $x$ which belongs to $K_{\epsilon}\cap \ball_p(0,m)$ for every $m\ge 1$ gives
		$$\sum_{x\in K_{\epsilon}\cap \ball_p(0,m)}\influence^{\epsilon}_x(A)\le 
		2\mathbb{E}_{p}(|\mathsf{Piv_A}|)+o(\epsilon^2).$$

		Taking the $\limsup $ and using Lemma \ref{russo} leads to
		$$ 
		\frac{\text{d}\mathbb{P}_p(A)}{\text{d}p}\ge \frac{1}{2}\limsup_{\epsilon\rightarrow 0}\sum_{x\in K_{\epsilon}\cap \ball_p(0,m)}\influence^{\epsilon}_x(A).$$
		
	Since the inequality holds for all $m\ge 1$, we can take $m=4n$ and derive an estimate for the influence of points in $K_{\epsilon}\cap \ball_p(0,4n)^c$.
		
	For every $x\in K_{\epsilon}\cap \ball_p(0,4n)^c$, if $\mathbf{1}_A(\mathcal{Z})\neq \mathbf{1}_A(\tilde{\mathcal{Z}})$,  $\mathcal{Z}\cup \mathcal{\tilde{Z}}$ must have at least one point in $\annsim_x$, and there exists at least one point in $(\mathcal{Z}\cup \mathcal{\tilde{Z}})\cap \annsim_x$ with its cell intersecting $\ball_p(0,n)$, i.e., $D_n(\mathcal{Z})\cap \ball_p(0,d_p(0,x))^c\neq \varnothing$. Therefore, combining \eqref{4times} and independence of Poisson point process in disjoint regions implies
	$$\begin{aligned}
	\influence^{\epsilon}_x(A)
	=&\ \mathbb{P}_{p}(\mathbf{1}_A(\mathcal{Z})\neq \mathbf{1}_A(\tilde{\mathcal{Z}}))\\
	\le&\ \mathbb{P}_p(|(\mathcal{Z}\cup \mathcal{\tilde{Z}})\cap \annsim_x|\ge 1)\mathbb{P}_p(\mathcal{Z}\cap \ball_p(0,\frac{1}{4}d_p(0,x))=\varnothing||(\mathcal{Z}\cup \mathcal{\tilde{Z}})\cap \annsim_x|\ge 1)\\
	=&\ \mathbb{P}_p(|(\mathcal{Z}\cup \mathcal{\tilde{Z}})\cap \annsim_x|\ge 1)\mathbb{P}_p(\mathcal{Z}\cap \ball_p(0,\frac{1}{4}d_p(0,x))=\varnothing)\\
	\le&\  c_2(\sinh \epsilon)^2\exp(-4\pi \lambda\sinh(\frac{d_p(0,x)}{8})^2).
	\end{aligned} $$
	Here, $c_2$ is a constant not depending on $\epsilon$.
	The proof then follows from summing over all $x\in K_{\epsilon}\cap \ball_p(0,4n)^c$ and letting $\epsilon$ go to zero.
	\end{proof}
	
	\subsection{Construction of the algorithm}
	
	Based on the space $(\prod_{x\in K_{\epsilon}}\annsim_x,\mathbb{P}_p)$ introduced in Section 3.1, we are now going to construct an algorithm $\algorithm_k$ determining $\mathbf{1}_{0\leftrightarrow \sph_p(0,n)}$ with its revealment having the following upper bound.

	\begin{lemma}
		\label{algorithmbdd}
		There exists a constant $c_0$ only depending on $p$ and $\lambda$ such that for any $k\in \{0,\cdots, n\}$, we can find an algorithm $\algorithm_k$ determining $f=\mathbf{1}_{0\leftrightarrow \sph_p(0,n)}$ such that 
		\begin{equation}\label{revealment}
			\delta_x(\algorithm_k)\le c_0\mathbb{P}_{p}(x\longleftrightarrow \sph_p(0,k)).
		\end{equation}
	\end{lemma}
	
	\begin{proof}
		
	We first define an auxiliary algorithm $\discover_x$ helping us reveal the colour of $\annsim_x$. 
	
	The idea of the algorithm $\discover_x$ is to check the random variables $\mathcal{Z}_y$ $(y\in K_{\epsilon})$ near the fixed point $x$ until the colour of every point in $\bar{\annsim}_x$ is known. 
	To put it precisely, we first set a parameter $k=0$. 
	When $k=l$ ($l$ is a positive integer), if the colour of all points inside $\bar{\annsim}_x$ are determined, the algorithm terminates and returns the colours of points in $\bar{\annsim}_x$ as the output of $\discover_x$. Otherwise, the algorithm checks the value of $\mathcal{Z}_y$ for $y\in K_{\epsilon}\cap \{y:d_p(y,x)\le k\}$ and set $k=l+1.$
	
	And the algorithm $\algorithm_k$ is constructed as follows:
	
	
	Set $X_0=\varnothing, Z_0=\sph_p(0,k), W_0=\varnothing, H_0=\varnothing, K_0=\varnothing, M_0=\varnothing$, and use $I=K_{\epsilon}\cap \ball_p(0,n+1)$ to include the set of points we may need to check.
	
	At every step $t$, assume that $X_t\subset K_{\epsilon}$ and $Z_t\subset \mathbb{H}^2$ have been determined.
	Let $M_{t+1}=\{x\in I\setminus X_t:\bar{\annsim}_x\cap Z_t\neq \varnothing\}$. If $M_{t+1}=\varnothing$, the algorithm terminates.
	If $M_{t+1}\neq \varnothing$, pick $y\in M_{t+1}$ and the algorithm runs the following steps:
		
	\begin{enumerate}
			\item Run $\discover_y$.
			\item Define $X_{t+1}=X_t\cup \{y\}$.
			\item Let 
			$$\begin{aligned}
			&H_{t+1}=\{\text{all\ the\ black\ points\ in\ }\omega\cap \annsim_y\ \text{connected to\ }Z_t\},\\
			&K_{t+1}=\{\text{all\ the\ black\ points\ in\ }\omega\cap \annsim_y\ \text{not\ connected to\ }Z_t\}.
			\end{aligned}$$
			\item Set $L_{t+1}=\{\text{all\ the\ points\ in\ }W_t\ \text{connected to\ }Z_t\cup H_{t+1}\}.$ Here, $L_{t+1}$ denote those black points in the first $t$ steps which are not discovered to be connected to $\sph_p(0,k)$ until the $t+1$ step runs.
			\item Set 
			$$\begin{aligned}
			&Z_{t+1}=Z_t\cup H_{t+1}\cup L_{t+1},\\
			&W_{t+1}=(W_t\setminus L_{t+1})\cup K_{t+1},
			\end{aligned}$$
			where $W_{t+1}$ represents those black points in the first $t+1$ steps that,  based on what are revealed now, we cannot determine whether they are connected to $\sph_p(0,k)$.
		\end{enumerate}
	
	The algorithm $\algorithm_k$ clearly determines $\mathbf{1}_{0\leftrightarrow \sph_p(0,n)}$ since it actually reveals all the black connected components of $\sph_p(0,k)$ in $\omega \cap \ball_p(0,n)$. Furthermore, every auxiliary algorithm $\discover_x$ only needs to discover finite points almost surely due to the integrability of $\mathbb{E}_p(|D_n(\mathcal{Z})|)$, therefore the algorithm $\algorithm_k$ will terminate in a finite time. 
	
	And we only need to prove the desired inequality \eqref{revealment}, which is trivial for $x\in K_{\epsilon}\setminus \ball_p(0,n+1)$ since those points will never be discovered in $\algorithm_k$. For $x\in K_{\epsilon}\cap \ball_p(0,n+1)$, if $x$ is revealed during the process, there exists $y\in \bar{\annsim}_x$ such that $y\longleftrightarrow \sph_p(0,k)$. For every representative point $x$ we can first fix a $z_x\in \mathbb{H}^2$ satisfying $\bar{\annsim}_x\subset \ball_p(z_x,1)$. Then $x$ is revealed during the process implies $\bar{\annsim}_x\leftrightarrow \sph_p(0,k)$.
	

	Combining Lemma \ref{FKG} and $p\in [\delta,1-\delta]$ we have
	$$\begin{aligned}
	&\delta_x(\algorithm_k)=\mathbb{P}_{p}(\bar{\annsim}_x \leftrightarrow \sph_p(0,k)) =\frac{\mathbb{P}_p(x\leftrightarrow \sph_p(0,k))}{\mathbb{P}_{p}(x\leftrightarrow \sph_p(0,k)|\bar{\annsim}_x \leftrightarrow \sph_p(0,k))}\\
	\le &\ \frac{\mathbb{P}_p(x\leftrightarrow \sph_p(0,k))}{\mathbb{P}_{p}(S_x \text{\ all\ black\ }|\bar{\annsim}_x \leftrightarrow \sph_p(0,k))}\\
    \overset{\eqref{FKG}} \le &\ \frac{\mathbb{P}_p(x\leftrightarrow \sph_p(0,k))}{\mathbb{P}_{p}(S_x \text{\ all\ black\ })}
	\le\frac{\mathbb{P}_p(x\leftrightarrow \sph_p(0,k))}{\mathbb{P}_{\delta}(\ball_p(z_x,1)\text{\ all\ black\ })}.
	\end{aligned}$$
	
	Denote $\frac{1}{c_0}=\mathbb{P}_{\delta}(\ball_p(z_x,1)\text{\ all\ black})$, then for every $k\in\{0,\cdots, n\}$, the constructed algorithm $\algorithm_k$ satisfies 
	$\delta_x(\algorithm_k)\le c_0\mathbb{P}_{p}(x\longleftrightarrow \sph_p(0,k)).$
	\end{proof}

	\subsection{Proof of Theorem 1}
	
	Before proving Theorem \ref{thm1}, we first introduce a lemma that converts the proof of Theorem \ref{thm1} to showing a family of differential inequalities.
	\begin{lemma}
		\label{lem6}
		Consider a converging sequence of increasing differential function $f_n:[\alpha_0, \alpha_1]\rightarrow [0, M]$ such that	for all $n\ge 1$, $$f_n'\ge \frac{n}{\Sigma_n}f_n ,$$
		where $\Sigma_n=\sum_{k=0}^{n-1}f_k$.
		Then, there exists $x_1\in [\alpha_0, \alpha_1]$ such that:\\
		1. For any $x<x_1$, there exists $c_x>0$ such that for any large $n$ , $f_n(x)\le \exp(-c_xn).$\\
		2. For any $x>x_1$, let $f(x)=\lim\limits_{n\to +\infty}f_n(x)$, then $f(x)\ge x-x_1$.
	\end{lemma}
	The proof of this lemma can be found in [\cite{sharpv}, Lemma 3.1]. 
		
	\begin{proof}[Proof of Theorem 1]
		
	By Lemma \ref{lem6}, it suffices to show that for every $0<\delta<\frac{1}{2}$, there exists $c=c(\delta)>0$ such that $\forall n\ge 1,\ p\in [\delta, 1-\delta]$, $S_n(p):=\sum_{k=0}^{n-1}\theta_k(p)$, 
	\begin{equation} \label{ineq}
	\theta_n'(p)\ge \frac{cn}{S_n(p)}\theta_n(p).
	\end{equation}
	Indeed, combining the fact that $0<p_c<1$ [\cite{BS01}, Theorem 1.5], the result follows from applying Lemma \ref{lem6} to $\alpha_0:=\delta<p_c<1-\delta=:\alpha_1$ and $f_n=\frac{\theta_n}{c}$.
		
	Now let $f=\mathbf{1}_{0\leftrightarrow \sph_p(0,n)}$, $\algorithm_k$ is an algorithm determining the function $f$, applying the OSSS inequality \eqref{osss} then leads to
	\begin{equation}\label{osssapply}
		\theta_n(p)(1-\theta_n(p))\le \sum_{x\in K_{\epsilon}}\delta_x(\algorithm_k)\influence^{\epsilon}_x(0\leftrightarrow \sph_p(0,n)).
	\end{equation}
	
	For every $\delta\in (0,\frac{1}{2})$, $p\in [\delta,1-\delta]$, it is easy to see that
	\begin{equation}\label{boundfortheta}
		\theta_n(\delta)\le \theta_n(p)\le \theta_n(1-\delta)\le \theta_1(1-\delta).
	\end{equation}

	Combining inequalities \eqref{revealment} and \eqref{boundfortheta} yields
	$$\begin{aligned}
	\theta_n(p)\le&\ \frac{1}{1-\theta_1(1-\delta)}\sum_{x\in K_{\epsilon}}\delta_x(\algorithm_k)\influence^{\epsilon}_x(0\leftrightarrow \sph_p(0,n))\\
	\overset{\eqref{revealment}}\le&\ \frac{c_0}{1-\theta_1(1-\delta)}\sum_{x\in K_{\epsilon}}\mathbb{P}_{p}(x\leftrightarrow \sph_p(0,k))\influence^{\epsilon}_x(0\leftrightarrow \sph_p(0,n)).
	\end{aligned}$$

	Summing the above inequality from $1$ to $n$ and dividing by $n$ gives
	\begin{equation}\label{tttt}
		\theta_n(p)\le \frac{c_3}{n}\sum_{x\in K_{\epsilon}}\sum_{k=1}^n\mathbb{P}_{p}(x\leftrightarrow \sph_p(0,k))\influence^{\epsilon}_x(0\leftrightarrow \sph_p(0,n)).
	\end{equation}

	We also have 
	\begin{equation}\label{sumbound}
		\sum_{k=1}^n\mathbb{P}_{p}(x\leftrightarrow \sph_p(0,k))\le \sum_{k=1}^n\theta_{d(x,\sph_p(0,k))}(p)\le 2\sum_{k=1}^n\theta_k(p)=2S_n(p).
	\end{equation}

	Taking $c=\frac{1}{4c_3}$, using \eqref{tttt},  \eqref{sumbound} and \eqref{derivative}, we then have for all $ p\in [\delta,1-\delta], n\ge 1,$
	$$\begin{aligned}
	\theta_n(p)
	&\overset{\eqref{sumbound}}\le \frac{2c_3}{n}S_n(p)\sum_{x\in K_{\epsilon}}\influence^{\epsilon}_x(0\leftrightarrow \sph_p(0,n))\\
	&\overset{\eqref{derivative}}\le \frac{4c_3}{n}S_n(p)\times \frac{\text{d}\mathbb{P}_{p}(0\leftrightarrow \sph_p(0,n))}{\text{d}p}=\frac{1}{cn}S_n(p)\theta_n'(p), \qquad 
	\end{aligned}$$
	which corresponds to the form of inequality \eqref{ineq} and therefore concludes the proof.
	
	\end{proof}

	\bibliographystyle{alpha}
	\bibliography{ref}

\begin{thebibliography}{DCRT19b}

\bibitem[AB87]{AB87}
Michael Aizenman and David~J. Barsky.
\newblock Sharpness of the phase transition in percolation models.
\newblock {\em Communications in Mathematical Physics}, 108(3):489--526, 1987.

\bibitem[AB18]{AB18}
Daniel Ahlberg and Rangel Baldasso.
\newblock Noise sensitivity and voronoi percolation.
\newblock {\em Electronic Journal of Probability}, 23:1--21, 2018.

\bibitem[ABGM14]{ABGM14}
Daniel Ahlberg, Erik Broman, Simon Griffiths, and Robert Morris.
\newblock Noise sensitivity in continuum percolation.
\newblock {\em Israel Journal of Mathematics}, 201(2):847--899, 2014.

\bibitem[AdlRG21]{ADG21}
Daniel Ahlberg, Daniel de~la Riva, and Simon Griffiths.
\newblock On the rate of convergence in quenched voronoi percolation.
\newblock {\em arXiv:2103.01870}, 2021.

\bibitem[AGMT16]{AGMT16}
Daniel Ahlberg, Simon Griffiths, Robert Morris, and Vincent Tassion.
\newblock Quenched voronoi percolation.
\newblock {\em Advances in Mathematics}, 286:889--911, 2016.

\bibitem[ATT18]{ATT17}
Daniel Ahlberg, Vincent Tassion, and Augusto Teixeira.
\newblock Sharpness of the phase transition for continuum percolation in
  $\mathbb{R}^2$.
\newblock {\em Probability Theory and Related Fields}, 172(1):525--581, 2018.

\bibitem[BH57]{bernoulli}
Simon~R Broadbent and John~M Hammersley.
\newblock Percolation processes: I. crystals and mazes.
\newblock In {\em Mathematical proceedings of the Cambridge philosophical
  society}, volume~53, pages 629--641. Cambridge University Press, 1957.

\bibitem[BR06a]{BR06}
B{\'e}la Bollob{\'a}s and Oliver Riordan.
\newblock The critical probability for random voronoi percolation in the plane
  is 1/2.
\newblock {\em Probability theory and related fields}, 136(3):417--468, 2006.

\bibitem[BR06b]{bpercolation}
B{\'e}la Bollob{\'a}s and Oliver Riordan.
\newblock {\em Percolation}.
\newblock Cambridge University Press, 2006.

\bibitem[BS01]{BS01}
Itai Benjamini and Oded Schramm.
\newblock Percolation in the hyperbolic plane.
\newblock {\em Journal of the American Mathematical Society}, 14(2):487--507,
  2001.

\bibitem[DCRT19a]{sharpv}
Hugo Duminil-Copin, Aran Raoufi, and Vincent Tassion.
\newblock Exponential decay of connection probabilities for subcritical voronoi
  percolation in $\mathbb{R}^d$.
\newblock {\em Probability Theory and Related Fields}, 173(1-2):479--490, 2019.

\bibitem[DCRT19b]{sharpr}
Hugo Duminil-Copin, Aran Raoufi, and Vincent Tassion.
\newblock Sharp phase transition for the random-cluster and potts models via
  decision trees.
\newblock {\em Annals of Mathematics}, 189(1):75--99, 2019.

\bibitem[DCT16]{history2}
Hugo Duminil-Copin and Vincent Tassion.
\newblock A new proof of the sharpness of the phase transition for bernoulli
  percolation and the ising model.
\newblock {\em Communications in Mathematical Physics}, 343(2):725--745, 2016.

\bibitem[Gri99]{grimmett2}
Geoffrey Grimmett.
\newblock {\em Percolation}.
\newblock Springer Berlin Heidelberg, 1999.

\bibitem[HM21a]{hbvoronoi}
Benjamin~T Hansen and Tobias M{\"u}ller.
\newblock The critical probability for voronoi percolation in the hyperbolic
  plane tends to $1/2$.
\newblock {\em Random Structures $\&$ Algorithms}, 2021.

\bibitem[HM21b]{hansennew}
Benjamin~T. Hansen and Tobias Müller.
\newblock Poisson-voronoi percolation in the hyperbolic plane with small
  intensities, 2021.

\bibitem[LP17]{PPP}
G{\"u}nter Last and Mathew Penrose.
\newblock {\em Lectures on the Poisson process}, volume~7.
\newblock Cambridge University Press, 2017.

\bibitem[LPY21]{stoppingsets}
G{\"u}nter {Last}, Giovanni {Peccati}, and D.~{Yogeshwaran}.
\newblock {Phase transitions and noise sensitivity on the Poisson space via
  stopping sets and decision trees}.
\newblock {\em arXiv:2103.01870}, 2021.

\bibitem[Men86]{Men86}
Mikhail~V Menshikov.
\newblock Coincidence of critical points in percolation problems.
\newblock In {\em Soviet Mathematics Doklady}, volume~33, pages 856--859, 1986.

\bibitem[MR96]{MR96}
Ronald Meester and Rahul Roy.
\newblock {\em Continuum Percolation}.
\newblock Cambridge University Press, 1996.

\bibitem[OSSS05]{osss}
Ryan O'Donnell, Mike Saks, Oded Schramm, and Rocco Servedio.
\newblock Every decision tree has an influential variable.
\newblock In {\em Proceedings of the 46th Annual Symposium on Foundations of
  Computer Science (FOCS)}, 2005.

\end{thebibliography}
\end{document}